\documentclass[10pt,a4]{article}

\usepackage[utf8]{inputenc}

\usepackage{verbatim} 

\usepackage{graphicx}
\usepackage{caption}
\usepackage{subcaption}
\usepackage{amssymb}
\usepackage{ textcomp } 
\usepackage{color}
\usepackage{amsmath}
\usepackage{amsthm}
\usepackage{enumerate}
\usepackage{hyperref}
\usepackage{multicol}

\newcommand{\R}{\mathbb{R}}

\newcommand{\abs}[1]{\left\vert #1 \right\vert}
\newcommand{\norm}[1]{\left\Vert #1 \right\Vert}
\newcommand{\set}[1]{\left\lbrace #1\right\rbrace}
\newcommand{\sse}{\subseteq}
\newcommand{\sprod}[1]{\left\langle #1 \right\rangle}

\newcommand{\prb}[1]{\mathbb{P}\left( #1 \right)}
\newcommand{\erw}[1]{\mathbb{E}\left( #1 \right)}

\usepackage{dsfont}

\newcommand{\1}{\text{1}}
\newcommand{\calS}{\mathcal{S}}

\newcommand{\dist}{\text{dist}}
\newcommand{\sgn}{\text{sgn}}
\newcommand{\supp}{\text{supp}}
\DeclareMathOperator{\argmin}{argmin}
\newcommand{\pos}{\text{pos}}
\newcommand{\diag}{\text{diag}}

\newtheorem{lem}{Lemma}
\newtheorem{prop}[lem]{Proposition}
\newtheorem{theo}[lem]{Theorem}
\newtheorem{cor}[lem]{Corollary}

\newtheorem{defi}[lem]{Definition}

\usepackage{tikz}
\usepackage{lipsum}

\newcommand\copyrighttext{%
  \footnotesize \copyright IEEE 2016. This paper has been published on IEEE Xplore, DOI \href{http://dx.doi.org/10.1109/tit.2016.2569122}{{\tt 10.1109/tit.2016.2569122}}}
\newcommand\copyrightnotice{%
\begin{tikzpicture}[remember picture,overlay]
\node[anchor=south,yshift=65pt] at (current page.south) {\fbox{\parbox{\dimexpr\textwidth-\fboxsep-\fboxrule\relax}{\copyrighttext}}};
\end{tikzpicture}%
}

\title{ Optimal Choice of Weights for \\ Sparse Recovery With Prior Information}
\author{Axel Flinth\thanks{ This work was supported by the Deutsche Forschungsgemeinschaft, DFG
under Grant KU 1446/18-1.
 The author is with the Institut f\"{u}r Mathematik, Technische Universit\"{a}t Berlin, Germany. (e-mail: flinth@math.tu-berlin.de).
  }
 }

\begin{document}

\maketitle
\begin{abstract}
Compressed sensing deals with the recovery of sparse signals from linear measurements. Without any additional information, it is possible to recover an $s$-sparse signal using $m \gtrsim s \log(d/s)$ measurements in a robust and stable way. Some applications provide additional information, such as on the location of the support of the signal. Using this information, it is conceivable the threshold amount of measurements can be lowered. A proposed algorithm for this task is \emph{weighted $\ell_1$-minimization}. Put shortly, one modifies standard $\ell_1$-minimization by assigning different weights to different parts of the index set $[1, \dots d]$. The task of choosing the weights is however non-trivial. 

This paper provides a complete answer to the question of an optimal choice of the weights. In fact, it is shown that it is possible to directly calculate unique weights that are optimal in the sense that the threshold amount of measurements needed for exact recovery is minimized. The proof uses recent results about the connection between convex geometry and compressed sensing-type algorithms.

\underline{Keywords:} Compressed sensing, Gaussian matrices, weighted basis pursuit, convex geometry.
\end{abstract}

\section{Introduction}
	
	At the heart of the theory of compressed sensing is the paradigm that it is possible to recover a sparse vector $x_0 \in \R^d$ using few linear measurements. A widely used method to perform this reconstruction is Basis Pursuit, i.e. to solve the minimization problem
	\begin{equation}
		\min \norm{x}_1 \text{ subject to } Ax=b. \tag{$P_1$}
\end{equation}	 
A fundamental result is that $ m \gtrsim s \log{d/s}$ Gaussian random measurements are sufficient for Basis Pursuit to succeed at recovering an $s$-sparse vector with overwhelmingly high probability\cite{CandesRombergTao2006}.

	The main problem about recovering a signal in this setting is in fact to decide where its support $\supp \ x_0$ is situated. Indeed, if we knew the exact position of $\supp \ x_0$, we would be able to recover $x_0$ by simply solving a set of linear equations which have more equations than unknowns, provided $m\geq s$. It is thus plausible that if we are given prior information about the location of $\supp \ x_0$, it should be possible to use this information to lower the threshold amount of measurements needed to secure recovery.  To be precise, let us assume that we know that a fraction $\alpha$ of the indices in some set $S$ are in the support of $x_0$. This can be interpreted as the probability for an $i \in S$ to be an element of $\supp \ x_0$. We will subsequently denote $T=\supp \ x_0$, for the sake of brevity. Intuitively, it seems that if $\alpha$ is high while $\rho := \abs{S}/d$ is not too large, we should be able to find the solution using less equations than if we did not have the initial guess.
	
 \copyrightnotice	
	
	This exact situation was investigated in the paper \cite{FriedMansSaabYil2012}. Their proposed solution  is to use a weighted $\ell_1$-approach, namely to solve the minimization problem
	\begin{equation}
		\min \norm{x}_{1,w} \text{ subject to } Ax=b  \tag{$P_1^w$},
	\end{equation}	
	where $\norm{\cdot}_{1,w}$ denotes the weighted $\ell_1$-norm, given by
	\begin{align*}
		\norm{x}_{1,w}= \sum_{i=1}^d w_i \abs{x(i)} 
	\end{align*}
	\big($\norm{\cdot}_{1,w}$ should of course not be confused with the weak $\ell_1$-norm.\big) We can likewise define the reweighted $\ell_p$-norms for every $p \in [1, \infty]$;
	\begin{align*}
		\norm{x}_{p, w} = \left(\sum_{i=1}^d w_i \abs{x(i)}^p\right)^\frac{1}{p} \ , \ p<\infty \text{ and } 
		\norm{x}_{\infty, w} = \sup_{i}  w_i \abs{x(i)}.
	\end{align*}
	
	Choosing the weights $w_i$ is a subtle process -- but the general idea is that the weights on the indices $S$ should be lower. This will have the effect that high values $x(i)$ on $S$ are not penalized, which is desirable if $T \approx S$.  Many papers have been concerned with the choice $w= \1_{S^c}$ (the setting corresponding to this choice is sometimes called \emph{modified Compressed Sensing}  \cite{VaswaniLu2009}), but the authors of \cite{FriedMansSaabYil2012} suggest that one should give oneself the freedom to choose $w$ as $\omega \1_S + \1_{S^c}$ for some $\omega \in [0,1]$. Here, as well as in the following, $\1_{M}$ refers to the indicator function of a set $M$, i.e. $$ \1_M(i) = \begin{cases} 1 \ \text{ if } i \in M \\ 0 \ \text{ else. }\end{cases}$$
	
 Tedious arguments by the authors of \cite{FriedMansSaabYil2012} showed that provided that the guess is good ($\alpha>.5$), the required bound on the RIP-constant for the matrix $A$ needed for robust recovery can be softened, and the stability constants get smaller, if one chooses $\omega =0$. The formulas actually suggest that one either should use $\omega=0$ or $\omega=1$, and never something in between. Numerical experiments do not harmonize with this rule of thumb: in fact, if the guess is bad ($\alpha<.5$), one does significantly better choosing $\omega$ between $0$ and $1$ than choosing one of the extreme values. The authors left it as an open problem to analyze how to choose the weight in an optimal way.

 \subsection{Previous results on the optimal choice of weights.}
 It is of course not entirely clear what ``optimal`` means in this context. One way of defining it is to say that a weight is optimal if the minimal amount of measurements needed in order for $(P_1^w)$ with $b=Ax_0$ to recover $x_0$ is as small as possible. A popular model assumption is that $A$ is chosen according to some probability distribution.  Because of its universality as a limit distribution, the Gaussian probability distribution is probably the most canonical one. In this setting, we search for the minimal amount of measurements required for $P_1^w$ to succeed with high probability (the exact meaning of ``high`` varies from application to application).
 
 The first result concerning choosing $\omega$ optimally in this sense was probably Theorem 4.3.3 of the PhD-thesis \cite{Xu2010} (the result was also presented in \cite{XuetAl2009}). Based on a very sophisticated argument relying on calculating internal and external angles of certain polytopes, the authors implicitly provide a threshold $m_0(\omega)$ so that if one uses $m \geq m_0(\omega)$ Gaussian measurements, $(P_1^w)$ will succeed with high probability. One can then minimize $m_0(\omega)$ with respect to the weight $\omega$. Since the formulas are very complicated, this is however a tough task.
 
 A result which is simpler to grasp is given in \cite{OymakHassibiEtAl2012}. The strategy is basically the same. One computes a threshold depending on the weight $\omega$, using the famous ``escape through a mesh lemma`` \cite{Gordon1988} and then minimizes the threshold with respect to $\omega$. Although the paper provides powerful and beautiful results on how the optimal threshold is depending on $\alpha$ and $\rho$, it does not  directly provide any method of actually calculating the optimal weights. Furthermore, as the authors point out, the result is only about an upper bound on how many measurements one needs. 
 
	A different analysis conducted in \cite{MansourSaab2014}. The authors define a \emph{weighted Null Space Property} and prove several interesting result about that property. In particular, they calculate upper bounds on thresholds on the number of measurements needed to secure that Gaussian matrices have that property. The main technical tool is again Gordon's escape through a mesh lemma. As a concrete weighting strategy, it is proposed to set the weight $\omega=1-\alpha$, since that weight minimizes their bound on the thresholds. One should note that the weighted Null Space property is uniform in nature in the sense that it is equivalent to that all signals supported on a set $S$ will be recovered from its linear measurements through weighted $\ell_1$-minimization as long as the quality of the support estimate is good enough - its exact position does not matter.

 \subsection{Contributions of this paper}
 
 In this paper, we will consider the full model with $k$ subsets $S_i$ of the complete index set $[1,\dots,d]$, each one with a corresponding probability $\alpha_i$. Correspondingly, we will allow for $k$ weights $\omega_i$, one on each $S_i$. We will use the short-hand notation $\omega=(\omega_1, \dots, \omega_k)$. The strategy of finding the optimal weights is very similar to the ones of the papers mentioned above -- we will calculate a threshold $m_0$ depending on $\omega$ and then prove that $m_0(\omega)$ is minimized for an optimal set of weights. To be precise, we will in fact not minimize $m_0$ itself, but only an upper and lower bound on $m_0$. These bounds are however simultaneously minimized for the same $\omega^*$ and also lie very close to each other. We will also investigate how the mentioned bounds on $m_0$ depend on the properties of the sets $S_i$, $i=1, \dots, k$. In particular, we will find the perhaps somewhat remarkable fact that one needs in principle as many measurements to recover an $(\alpha_1 \abs{S_1} +\dots + \alpha_k \abs{S_k})$-sparse signal given the prior knowledge of the probabilities $\alpha=(\alpha_1, \dots, \alpha_k)$ and the sets $(S_1, \dots S_k)$ as recovering $k$ signals $x_i$ separately, where for each $i$, $x_i$ is $\alpha_i\abs{S_i}$-sparse and known to be supported on the set $S_i$. This was already proven for upper bounds on the measurements needed in \cite{OymakHassibiEtAl2012}, but here, we provide also a statement about lower bounds.
 
 In order to do this, we will use a powerful and recent result from \cite{AmelunxLotzMcCoyTropp2014}, namely that the recovery probability of a convex program undergoes a \emph{phase transition} as the number of measurements surpass the \emph{statistical dimension} $\delta(C)$ of a certain convex cone. We will provide a way of calculating a set of weights $\omega^*$ which simultaneously minimizes a tight \emph{lower and upper} bound of $\delta(C)$, and hence to some extent settle the problem regarding the optimal choice of weights when recovering a sparse vector $x_0$ with prior information. Additionally, we will even provide a simple analytical formula for $\omega^*$. 
 
 The fact that we are using the new notion of statistical dimensions is the main technical difference from the prior work mentioned in the previous section, where the escape through a lemma was the main tool in most cases. \newline
 
 The notation in the paper will be fairly standard. There are however a few things that should be pointed out. For a vector $w \in \R^d$, $\diag(w)$ is the diagonal matrix in $\R^{d,d}$ whose diagonal is $w$. $\pos(x)$ is the positive part of a real number, i.e. $\pos(x)=x$ if $x \geq 0$, and otherwise $\pos(x)=0$. For $x \in \R^d$, $\sgn \ x$ denotes the vector consisting of the signs of the entries of $x$, with the convention $\sgn \ 0 = 0$. Finally, $\abs{S}$ will have different meanings depending on its argument. If $S$  is a real number, $\abs{S}$ denotes its absolute value. For vectors, $\abs{S}$ is the vector formed by taking the absolute value pointwise. If $S$ is a set, $\abs{S}$ denotes the number of elements of $S$.
 
 \section{Statistical dimensions and the measurement threshold}
	
	Let us begin by fixing the model situation and notation. $x_0 \in \R^d$ is an $s$-sparse vector, meaning that only $s$ of the entries $x_0(i), i =1, \dots d$ are unequal to zero. Additional to the prior information that $x_0$ is sparse, we are given some partial information where the support $T$ of $x_0$ is situated, expressed by the partition $\mathcal{S}=(S_i)_{i=1}^k$ of $[1,\dots d]$. With each set $S_i$ we associate the two parameters $\alpha_i$ and $\rho_i$
	\begin{align*}
		 \rho_i = \frac{\abs{S_i}}{d} \text{ and } \alpha_i = \frac{\abs{T \cap S_i}}{\abs{S_i}}, 
	\end{align*}
	reflecting its size and the probability for a $j \in S_i$ to be an element in $T$, $i =1, \dots, d$, respectively. To each $S_i$, we will assign a weight $\omega_i \geq 0$ and always choose the weights according to $$w = \sum_{i=1}^k \omega_i \1_{S_i}.$$ We will without loss of generality assume that $\alpha_i>0$ for every $i$ (if $\alpha_i=0$ for some $i$, the corresponding set of indices $S_i$ can be completely ignored in the recovery process.)
	
	Now we define the object which we will search for in the rest of this section: the threshold amounts of measurements needed for the program $(P_1^w)$ to be successful. 
	
	\begin{defi} \label{def:Thresholds}
		Let $A\in \R^{m,d}$ be a Gaussian matrix. Then $\mu_{s,d}$ denotes the normalized threshold amount of measurement needed for $(P_1)$ with $Ax_0=b$ to succeed at recovering an $s$-sparse vector $x_0$, i.e.,  $(P_1)$ will succeed with high probability if $m \geq d(\mu_{d,s} + o(1))$, and will fail with high probability if $m \leq d(\mu_{d,s} - o(1))$. 
		
		In the exact same way, we define for a given support $T$ and partition $\mathcal{S}=(S_i)_{i=1}^k$ the number $m_{\calS,T,w}$ as the threshold amount of measurements for $(P_1^w)$ with $Ax_0=b$ to succeed at recovering a vector $x_0$ supported on $T$. Finally, $m_{\calS,T}$ is the value obtained by choosing the weights $w$ optimally, i.e. $m_{\calS,T}=\inf_{w\geq 0} m_{\calS,T,w}$.
	\end{defi}
	
	Note that we are only concerned with exact recovery of exactly sparse signals, and, in particular, leave out cases where noise is involved and the signals are only approximately sparse. We believe that the requirement of stability and robustness will not significantly increase the measurement threshold, but also think that a thorough analysis of these questions are beyond the scope of this paper. Therefore they are left as possible lines of future research.

 The authors of  \cite{OymakHassibiEtAl2012}  provided upper bounds $\hat{\eta}_{s,d}$  and $\hat{\eta}_{\mathcal{S},T}$ of $\mu_{s,d}$ and $m_{\calS,T}$, respectively. (Note that our notation is a bit different to theirs: what we call $\hat{\eta}_{\mathcal{S},T}$ is denoted $\hat{\eta}(\alpha,\beta)$ by them, where $\alpha$ and $\beta$ denotes $\abs{S_1 \cap T}$ and $\abs{S_2\cap T}$, respectively.)
 They then proceed to prove the following very beautiful result.
 \begin{theo} \cite[Th. 3.1]{OymakHassibiEtAl2012} \label{th:SynthesisUpperThresholds} We have
 	\begin{align*}
 		\hat{\eta}_{(S_1, S_2),T} = \rho_1 \hat{\eta}_{\alpha_1\abs{S_1},\abs{S_1}} + \rho_2 \hat{\eta}_{\alpha_2\abs{S_2},\abs{S_2}}
 	\end{align*}
 \end{theo}
 
 Theorem \ref{th:SynthesisUpperThresholds} in fact states that in order to find an $\alpha_1 \rho_1 d + \alpha_2 \rho_2 d$-sparse vector using the prior information, we need about as many measurements as the sum of the measurements needed to separately recover two  $\alpha_i \rho_i d$-sparse vectors known to be supported on $S_i$, respectively. The authors of \cite{OymakHassibiEtAl2012} conjectured that the theorem can be generalized to more than $k=2$ sets $S_i$. We will do this at the end of the paper, with the addition that the upper bound we provide on $m_{\calS,T}$ in fact also induces a lower bound.

$\mu_{s,d}$ and $m_{\calS,T,w}$ can be calculated using a few results from \cite{AmelunxLotzMcCoyTropp2014}. Before stating the results we will use, let us define the statistical dimension of a cone $C$.
 
\begin{defi} \label{def:statDim} \cite{AmelunxLotzMcCoyTropp2014} Let $C \sse \R^d$ be a closed convex cone and $g$ be a Gaussian vector. Let furthermore $\Pi_C$ denote the metric projection onto $C$, i.e.
\begin{align*}
	\Pi_C(x) = \argmin_{c \in C} \norm{ c-x}_2.
\end{align*}
The \emph{statistical dimension} $\delta(C)$ of $C$ is then defined as
\begin{align*}
	\delta(C) = \erw{\norm{\Pi_C g}_2^2}.
\end{align*}
The statistical dimension of a general convex cone is the statistical dimension of its closure.
\end{defi} 

The main result of \cite{AmelunxLotzMcCoyTropp2014} shows that the statistical dimension of the descent cone $\mathcal{D}(f,x_0)$ of a convex function $f$ at $x_0$, i.e., 
\begin{align*}
	\mathcal{D}(f,x_0)= \set{v \in \R^d \vert \ \exists \ \tau >0 : f(x_0+\tau v) \leq f(x_0)},
\end{align*} 
can be used to calculate the minimal amount of measurements needed for the program
\begin{align}
	\min f(x) \text{ subject to } Ax =b  \tag{$P_f$}
\end{align}
with $b=Ax_0$ to succeed at recovering a vector $x_0 \in \R^d$.  It is well known that $(P_f)$ has $x_0$ as its solution if and only if $\mathcal{D}(f,x_0) \cap \ker A  = \set{0}$. This leads to the following precise result:
\begin{theo} \cite[Th. I]{AmelunxLotzMcCoyTropp2014} \label{th:statDim} Let $C \sse \R^d$  be a convex cone and $A\in \R^{m,d}$ a Gaussian matrix. Then if
	\begin{align*} 
		\delta(C) + \sqrt{8 \log(4/\eta)}\sqrt{d} \leq  m, 
		\end{align*}
		we have $\prb{ C \cap \ker A  \neq \set{0}} \leq \eta$. Likewise if
		\begin{align*}
			\delta(C) - \sqrt{8 \log(4/\eta)}\sqrt{d}  \geq  m,
			\end{align*}
			then $ \prb{ C \cap \ker A  \neq \set{0}} \geq 1-\eta$.
			
			In particular, the probability that the solution of the problem $(P_f)$ is equal to $x_0$ undergoes a phase transition as $m$ surpasses $\delta(\mathcal{D}(f,x_0))$, and hence we have
			\begin{align*}
				\mu_{s,d}= \frac{\delta\left(\mathcal{D}(\norm{\cdot}_1,x_0)\right)}{d} \text{ and } m_{\calS,T,w} = \frac{\delta\left(\mathcal{D}(\norm{\cdot}_{1,w},x_0)\right)}{d} .
			\end{align*}
\end{theo}

 The authors of \cite{AmelunxLotzMcCoyTropp2014} proceeded to calculate $\delta\left(\mathcal{D}(\norm{\cdot}_{1},x_0)\right)$ as an exemplary application of a technique they presented. We will use the same technique to calculate $\delta\big(\mathcal{D}(\norm{\cdot}_{1,w},x_0)\big)$, but before that, let us state the result regarding standard Basis Pursuit.
 
 \begin{theo} \cite[Prop. 4.5]{AmelunxLotzMcCoyTropp2014}. Define the function $J_\sigma$ through
	\begin{align*}
		J_\sigma(\tau)= \sigma (1 + \tau^2) + (1-\sigma) \varphi(\tau),
\end{align*}	 
 where $\varphi: [0,1] \to [0,1]$ is defined through
 \begin{align} \label{eq:varphifunc}
  \varphi(t) = \sqrt{\frac{2}{\pi}}\int_{t}^\infty (x-t)^2 \exp(-x^2/2)dx.
\end{align}  
The statistical dimension of the descent cone of the $\ell_1$-norm at an $s$-sparse $x_0 \in \R^d$ satisfies
\begin{align*}
 \inf_{\tau>0} J_{\sigma}(\tau) - \frac{2}{d}\sqrt{\frac{1}{\sigma}} \leq \frac{\delta\left(\mathcal{D}( \norm{\cdot}_{1}, x_0)\right)}{d} \leq \inf_{\tau>0} J_{\sigma}(\tau),
 \end{align*}
 where $\sigma =s/d$.
 \end{theo}
 
 Now as promised, we will prove a similar result regarding the weighted Basis Pursuit $(P_{1,w})$. We will use the same technique as the authors of \cite{AmelunxLotzMcCoyTropp2014}  and therefore start by calculating the expected length of the projection of a Gaussian vector onto a dilation of the subdifferential of the weighted $\ell_1$-norm. The \emph{subdifferential} \cite[p.214-15]{Rockafellar1970} of a convex function $f$ at a point $x_0$ is the set of vectors $p$ satisfying 
 \begin{align*}
 	 \quad f(x + v) \geq f(x_0) + \sprod{p,v} \text{ for all } v \in \R^d.
\end{align*}  
For this, we will need a fact on subdifferentials of norms. The following result is folklore in this area, but let us include a proof for completeness.
 
 \begin{prop} For any norm $\norm{\cdot}$ on $\R^d$, denote by $\norm{\cdot}_{*}$ its dual norm, i.e.
 \begin{align*}
 	\norm{p}_{*} = \sup_{\norm{v}=1} \sprod{p,v}.
\end{align*} 
 Then for every $x_0\neq 0$
	\begin{align*}
		\partial_{x_0}\norm{\cdot} = \set{ p \ \vert \ \norm{p}_* =1, \ \sprod{p, x_0} = \norm{x_0}},
	\end{align*}
	where $\sprod{p,x_0}$ denotes the canonical dual pairing $\R^d \times \R^d$, i.e., the Euclidean scalar product.
	\end{prop}
	\begin{proof}
		We aim to characterize the vectors $p \in \R^d$ which have the property
		\begin{align} \label{eq:normSubDiff}
			 \norm{x_0 +v} \geq \norm{x_0} + \sprod{p, v} \text{ for all } v \in \R^d
		\end{align}
		By taking the supremum in \eqref{eq:normSubDiff} over all $v$ with $\norm{v}=1$, we immediately obtain
		\begin{align*}
			\sup_{\norm{v}=1} \sprod{p, v} \leq \sup_{\norm{v}=1} \norm{x_0 +v} - \norm{x_0} \leq \sup_{\norm{v}=1} \norm{v} =1,
		\end{align*}
		i.e. $\norm{p}_* \leq 1$. By further choosing $v= -x_0$, we get $\sprod{p, x_0} \geq \norm{x_0}$, which proves $\norm{p}_*\geq 1$, and therefore also $\norm{p}_*=1$. This further implies that $\sprod{p, x_0} = \norm{x_0}$. If on the other hand $\norm{p}_*=1$ and $\sprod{p, x_0} = \norm{x_0}$, we have
		\begin{align*}
			\norm{x_0} + \sprod{p, v} = \sprod{p, x_0 + v} \leq \norm{x_0+v},
\end{align*}	
and hence 	 \eqref{eq:normSubDiff} is true, and we are done.
	\end{proof}
 
 Now we can calculate  $\erw{\dist^2(g, \tau \cdot \partial_{x_0} \norm{\cdot}_{1,w})} $ for $x_0 \in \R^d$ and a Gaussian vector $g$.
 
 \begin{lem} \label{lem:distSubdiffL1Weight}
 	Let $x_0 \in \R^d$ be arbitrary and $g$ a Gaussian vector. We have
 	\begin{align*}
 	  \erw{d^{-1}\dist^2(g, \tau \cdot \partial_{x_0} \norm{\cdot}_{1,w})} &= \sigma + \sum_{i=1}^k \rho_i(\alpha_i (\omega_i \tau)^2 
 	  + (1-\alpha_i) \varphi(\omega_i \tau)),
	\end{align*}
	where $\sigma = \norm{x_0}_0/d$, $w = \sum_{i=1}^k \omega_i \1_{S_i}$ and $\varphi$ is defined through Equation \eqref{eq:varphifunc}.
 \end{lem}
 
 \begin{proof}
 Let us start by noticing that the dual norm of $\norm{\cdot}_{1,w}$ is $\norm{\cdot}_{\infty,w^{-1}}$. The Hölder inequality implies that for every $x, v \in \R^d$ ,
 \begin{align*}
 	\sprod{x,v} = \sum_{j=1}^dw_j^{-1}x_jw_j v_j \leq \norm{(w_i^{-1}x_i)_{i=1}^d}_\infty \norm{(w_iv_i)_{i=1}^d}_1= \norm{x}_{\infty,w^{-1}}\norm{v}_{1,w},
\end{align*}  
i.e. $\norm{x}_{1,w}^* \leq \norm{x}_{\infty, w^{-1}}$. By choosing $v$ equal to the vector for which $v(i)= \sgn \ x(i)$ if $i$ is the index where $w(i)^{-1}x(i)$ is maximal, and else $v(i)=0$, we even see that equality holds.

This, together with the fact that under the assumption $\norm{p}_{w^{-1}, \infty}=1$ , $\sprod{x_0, p}=\norm{x_0}_{1,w}$ if and only if $p \vert_{\supp \ x_0} = D^w \sgn \ x_0 \vert_{\supp \ x_0}$, implies that 
 \begin{align*}
 	\partial_{x_0} \norm{ \cdot}_{1,w} = \set{ D^w \sgn \ x_0 + v \ \Big\vert \  \supp \ v \cap \supp \ x_0 = \emptyset,
 	 \forall j : \abs{v_j} \leq w_j },
 \end{align*}
 where $D^w = \diag(w) = \diag\left(\sum_{i=1}^k \omega_i \1_{S_i} \right)$. Abbreviating $T = \supp \ x_0$, we see that
 \begin{align*}
 	\dist^2(g, \tau \cdot \partial_{x_0} \norm{\cdot}_{1,w}) =& \sum_{j \in T} (g(j) -\sgn(x_0(j)) w_j\tau)^2 + \sum_{j \notin T} \pos(\abs{g(j)}- w_j\tau)^2
 \end{align*}
 (the part on $T$ of a vector in $\tau \partial_{x_0} \norm{\cdot}_{1,w}$ is fixed to $\tau D^w\sgn \ x_0$, and the other part can elementwise not be larger than $w$.) Taking the expectation, we obtain
 \begin{align*}
 	\mathbb{E}\big(\dist^2(g, \tau \cdot \partial_{x_0} \norm{\cdot}_{1,w})\big)  
 	=& \sum_{j\in T} (1 + w_j^2\tau^2) 
 	+ \sum_{j \notin T} \sqrt{\frac{2}{\pi}} \int_{w_j \tau}^\infty (x-w_j\tau)^2 \exp(-x^2/2)dx \\
 	 =& \abs{T}+  \sum_{i=1}^k\abs{S_i \cap T} \tau^2 \omega_i^2+ \abs{S_i \cap T^c}  \varphi(\omega_j\tau) \\
 	 =& d\left(\sigma + \sum_{i=1}^k \rho_i\left(\alpha_i (w_i \tau)^2  + (1-\alpha_i) \varphi(\omega_i \tau)\right)\right) .
 	 \end{align*}
 \end{proof}
 
For convinence, let us abbreviate
\[J_{\sigma, \alpha, \rho, \omega}(\tau)= \sigma + \sum_{i=1}^k \rho_i(\alpha_i (\omega_i \tau)^2 +
 	   (1-\alpha_i) \varphi(\omega_i \tau)).\]
 	Also define 
 	\begin{align} \label{eq:mtilde}	\tilde{m}_{\calS,T,w}:=\inf_{\tau>0}J_{\sigma, \alpha, \rho, \omega}(\tau),
 	\end{align} 
 	where $\alpha$ and $\rho$ are the parameters associated with $\calS$ and $T$, and $w=\sum_{i=1}^k \omega_i \1_{S_i}$.

 	 As the notation suggests, $\tilde{m}_{\calS,T,w}$  is not far away from $m_{\calS,T,w}$. We will prove this using Lemma \ref{lem:distSubdiffL1Weight} together with the following theorem from \cite{AmelunxLotzMcCoyTropp2014}. 
 	 
 	 \begin{theo}   \label{th:statDimBound}
 	 \cite[Th. 4.3]{AmelunxLotzMcCoyTropp2014}. Let $\norm{\cdot}$ be a norm on $\R^d$. The the statistical dimension of the descent cone of $\norm{\cdot}$ at a point $x_0 \in \R^d$ satisfies
 	 \begin{align}
 	0 \leq  \inf_{\tau>0} \erw{ \dist^2(g, \tau \partial_{x_0} \norm{\cdot})} - \delta(\mathcal{D}&(\norm{\cdot}, x_0))
 	\leq \frac{2 \sup_{s \in \partial_{x_0} \norm{\cdot}}{\norm{s}_2}}{\norm{\tfrac{x_0}{\norm{x_0}_2}}}.
 \end{align}
 \end{theo}
 
  \begin{prop} \label{cor:statDimWeightBound}
 	The statistical dimension of $\mathcal{D}( \norm{\cdot}_{1,w}, x_0)$ satisfies
		\begin{align} \label{eq:statDimTight}
			\tilde{m}_{\calS,T,w} - \frac{2}{d}\sqrt{\frac{1}{\bar{\alpha}}} \leq \frac{\delta\left(\mathcal{D}( \norm{\cdot}_{1,w}, x_0)\right)}{d} \leq \tilde{m}_{\calS,T,w},
		\end{align}
		where $\bar{\alpha}= \min_i{\alpha_i}$ and $\tilde{m}_{\calS,T,w}$ is defined through \eqref{eq:mtilde}. Moreover, the following bound independent of $\alpha$ holds:
		\begin{align} \label{eq:statDimSloppy}
		\tilde{m}_{S,T,w} - \frac{2}{\sqrt{d}} \leq \frac{\delta\left(\mathcal{D}( \norm{\cdot}_{1,w}, x_0)\right)}{d} \leq \tilde{m}_{S,T,w}.
		\end{align}
 \end{prop}
 
 \begin{proof}
 We use Theorem \eqref{th:statDimBound}. To control the error term, by using that $\abs{S_i \cap T} =  \alpha_i \rho_i  d$ for each $i$, we obtain the estimate
 \begin{align*}
 	\norm{x}_{1,w} &= \sum_{j \in T} w_j\abs{x_j} \leq \norm{x}_2\sqrt{\sum_{j \in T} w_j^2}= \sqrt{d\sum_{i=1}^k\rho_i\alpha_i\omega_i^2}  \norm{x}_2
 \end{align*}
 which is valid for any $x$ supported on $T$. We even have equality if we choose $\abs{x}$ parallel to $w \vert_{S_0}$. We can do this since, as pointed out in \cite{AmelunxLotzMcCoyTropp2014}, there is no need to use Theorem \eqref{th:statDimBound} directly for $x_0$. We only have to secure that  $x$ has the same $\mathcal{D}(\norm{\cdot}_{1,w},x)$ and $\partial_x \norm{\cdot}_{1,w}$ as $x_0$, since $\inf_{\tau>0} J_{\sigma, \alpha, \rho, \omega}(\tau) - \delta(\mathcal{D}(\norm{\cdot}_{1,w}, x_0))$ only depends on those entities. It is clear from the considerations in the proof of Lemma \ref{lem:distSubdiffL1Weight} that for this to be true, $x$ only has to be supported on $T$ and have the same sign pattern as $x_0$. It is equally clear that there exist such vectors $x$ so that $\abs{x}$ is parallel to $w \vert_{T}$.
 
 This proves that it is possible to choose the denominator of the error term in Theorem \eqref{th:statDimBound} equal to $\sqrt{d\sum_{i=1}^k\rho_i\alpha_i\omega_i^2}$. To bound the numerator, it suffices to notice that each vector $s \in \partial_{x_0} \norm{\cdot}_{1,w}$ satisfies
 \begin{align*}
 	\norm{s}_2 = \sqrt{ \sum_{j =1}^d s_j^2}\leq\sqrt{ \sum_{j =1}^d w_j^2} &= \sqrt{ \sum_{i=1}^k\sum_{j \in S_i} \omega_i^2} = \sqrt{d\sum_{i=1}^k \rho_i \omega_i^2},
 \end{align*}
 since $\abs{s_j}\leq w_j$ for each $j$, and $\abs{S_i}=\rho_i d$. To finish the proof of the tighter bound \eqref{eq:statDimTight}, we note that 
\begin{align*}
	\frac{\sqrt{d\sum_{i=1}^k \rho_i \omega_i^2}}{\sqrt{d\sum_{i=1}^k\rho_i\alpha_i\omega_i^2}} \leq \sqrt{\frac{d\sum_{i=1}^k \rho_i \omega_i^2}{d\bar{\alpha}\sum_{i=1}^k\rho_i\omega_i^2}} = \frac{1}{\sqrt{\bar{\alpha}}}.
\end{align*}
Finally, the bound $\eqref{eq:statDimSloppy}$ easily follows from that statement, since
\begin{align*}
	d\bar{\alpha} = \min_i d\frac{\abs{S_i \cap T}}{\abs{S_i}} \geq 1,
\end{align*}
using the facts that  $\abs{S_i \cap T} \geq 1$ and $\abs{S_i}\leq d$.
 \end{proof}
 

 \section{Optimizing $m_{S,T,w}$.}
 
	Proposition \ref{cor:statDimWeightBound} states that $m_{\calS,T,w}$ is bounded both below and above by the number $\tilde{m}_{\calS,T,w}=\inf_{\tau>0} J_{\sigma, \alpha, \rho, w}(\tau)$, up to an error which is independent of the weights $w$. In particular, if we denote $\tilde{m}_{\calS,T}= \inf_{w \geq 0}\tilde{m}_{\calS,T,w}$, we have
 \begin{align*}
 	\tilde{m}_{\calS,T} - \frac{2}{\sqrt{d}} \leq m_{S,T} \leq \tilde{m}_{\calS,T}.
 \end{align*}
 Hence, $\tilde{m}_{\calS,T}$ is equal to  $m_{\calS,T}$ up to an asymptotically vanishing error term. In this section, we prove that  $\tilde{m}_{\calS,T,w}$ has a unique minimum  at some $\omega^*$ for every $\calS$. As the proof will show, it can furthermore easily be found by minimizing a convex function in two variables. The proof uses several techniques exploited for partly other purposes in \cite{AmelunxLotzMcCoyTropp2014}.
 
\begin{theo} \label{prop:optimalWeight}
	Assume that $\abs{T} < d$ and let $x_0 \in \R^d$ be a vector supported on $T$. For each partition $\calS=(S_i)_{i=1}^k$ of $[1,\dots, d]$, there exists weights $\omega^* \in \R_{+}^k$ so that 
		\begin{align*}
			\tilde{m}_{\calS,T,\omega}=\inf_{\tau>0} \erw{\dist^2(g, \tau \partial_{x_0} \norm{\cdot}_{1,w})}
		\end{align*}
		is minimal, where $w=  \sum_{i=1}^k \omega_i \1_{S_i}$. The weights $\omega^*$ are for instance uniquely determined by the additional constraint $\norm{\omega^*}_\infty=1$.
\end{theo}

\begin{proof}
First, we note that the chain rule proves that
\begin{align*}
	\partial_{x_0} \norm{\cdot}_{1,w} &= D^w \left(\partial_{D^w x_0}. \norm{\cdot}_1\right) =D^w\left(\partial_{ x_0} \norm{\cdot}_1\right)=: D^w C,
\end{align*}
where we used ${D^w}^*=D^w$ and the fact that the subdifferential $\partial_{x}\norm{\cdot}_1$ only depends in the sign pattern of $x$ . Hence, since $x_0$ and $D^wx_0$ have the same sign pattern, $\partial_{D^w x_0} \norm{\cdot}_1 = \partial_{x_0} \norm{\cdot}_1$. If for $v \in \R_{+}^k$, we denote $\Delta^v = \diag( \sum_{i=1}^k v_i \1_{S_i})$, then $\tau \Delta^\omega = \Delta^{\tau \omega}$ and hence
\begin{align*}
	\inf_{\omega >0} \inf_{\tau>0} \erw{\dist^2(g, \tau \partial_{x_0} \norm{\cdot}_{1,w})}= \inf_{v >0} \erw{\dist^2(g, \Delta^v C)}.
\end{align*}
Therefore, if we prove that there exists a unique $v^*\neq 0$ so that $\erw{\dist^2(g, \Delta^v C)}$ is minimal, then $\omega^* = v^*/\norm{v^*}_\infty$ is the optimal weight we are looking for, since then 
\begin{align*}
	\inf_{\tau>0} \mathbb{E}\big(\dist^2(g, \tau \partial_{x_0} \norm{\cdot}_{1,w^*})\big) = \inf_{\tau>0} \erw{\dist^2(g, \Delta^{\tau \omega^*}C} = \erw{\dist^2(g, \Delta^{v^*} C)} \leq \erw{\dist^2(g, \Delta^{ \tau \omega} C)}
\end{align*} 
for any $\tau>0$, $\omega\geq 0$, from which the optimality of $\omega^*$ follows. $\omega^*$ of course denotes $\sum_{i=1}^k \omega^*_i \1_{S_i}$.

In order to prove that there exists a unique minimizer of $\phi(v) = \erw{\dist^2(g, \Delta^v C)}$ on $\R_{+}^k$, it suffices to prove that $\phi$ is strictly convex, continuous and coercive (``large outside a ball``).

	\emph{Strict convexity:} Let $s_i \in \Delta^{v_i}C$, $i=1,2$, and $\theta \in [0,1]$. Then (due to the structure of $C=\partial_{ x_0} \norm{\cdot}_1$), we know that $s_i = \Delta^{v_i} (\sgn \ x_0 + u_i)$ for some $u_i$ with $\supp \ u_i \cap \supp \ x_0 = \emptyset$ and $\norm{u_i}_\infty \leq 1$. This implies that
	\begin{align*}
		\theta s_1 + (1-\theta)s_2 =& (\theta \Delta^{v_1} 
		+ (1-\theta) \Delta^{v_2})\sgn \ x_0 + (\theta \Delta^{v_1} u_1+ (1-\theta) \Delta^{v_2}u_2).
	\end{align*}
	The right hand side of the previous equation is an element of $\Delta^{\theta v_1 + (1-\theta) v_2}C$. This since $\theta \Delta^{v_1} + (1-\theta) \Delta^{v_2}=\Delta^{\theta v_1 + (1-\theta) v_2}$, the action of diagonal matrices does not change the support of a vector, and the chain of inequalities
	 \begin{align*}
	 \big\vert \langle\theta \Delta^{v_1} u_1+ (1-\theta) \Delta^{v_2}u_2,e_i\rangle \big\vert 
	 & \leq \theta \abs{\sprod{\Delta^{v_1}u_1,e_i}} +
	  (1-\theta)\abs{\sprod{\Delta^{v_2}u_2,e_i}} \\
	  & = \theta \abs{u_1(i)}v_1(i) + (1-\theta)\abs{u_2(i)}v_2(i) \leq \theta v_1(i) + (1-\theta)v_2(i).
	 \end{align*} 
	 From the latter we deduce that $\theta \Delta^{v_1} u_1+ (1-\theta) \Delta^{v_2}u_2$ can be written as $\Delta^{\theta v_1 + (1-\theta)v_2}u$, where $\supp \ u \cap \supp \ x_0~=                                                                                                                                          \emptyset$ and $\norm{u}_\infty \leq 1$.
	 We have thus proven that for every triple $v_1, v_2 \in \R^k_{+}$, $\theta \in [0,1]$,
	\begin{align*}
		\theta \Delta^{v_1} C + (1-\theta) \Delta^{v_2} C \sse \Delta^{\theta v_1 + (1-\theta)v_2}C.
	\end{align*}
	Therefore, for each fixed $ u \in \R^d$, $s_i \in \Delta^{v_i}C$, $i=1,2$, we have
	\begin{align*}
		\dist(u, \Delta^{\theta v_1 + (1-\theta)v_2}C) \leq \dist(u,\theta \Delta^{ v_1}C + (1-\theta) \Delta^{v_2}C)& \leq \norm{u-(\theta s_1 + (1-\theta)s_2)}_2 \\
		&\leq \theta \norm{u- s_1}_2 + (1-\theta) \norm{u- s_2}_2
	\end{align*}
	Taking the infimum over $s_1$ and $s_2$, it follows that for each fixed $u$, the function $v \to \dist(u, \Delta^v C)$ is convex. The square of a non-negative convex function is still convex, and since choosing $u=g$ a Gaussian and taking the expectation also does not destroy the convexity, and we can conclude that $\phi$ is convex. 
	
	It remains to prove strict convexity. Since convexity already has been established, it suffices to prove that there does not exist $v_1 \neq v_2$ and $\theta \in (0,1)$ with $\phi(\theta v_1+(1-\theta)v_2)=\theta \phi(v_1) + (1-\theta)\phi(v_2)$. Towards a contradiction,  assume that this is not true. Then
	\begin{align*}
		\mathbb{E}\Big(\dist^2&(g, \Delta^{\theta v_1 + (1-\theta)v_2}C)\Big) = \erw { \theta \dist^2(g, \Delta^{v_1}C)+ (1-\theta)\dist(g, \Delta^{v_2}C)}.
	\end{align*}
	According to what was just proven, the expression over which we are taking the expectation on the right hand side is not smaller than the expression on the left hand side. Hence, in order for equality to be true, the two expressions have to be equal almost surely. But there exists some $u \in \Delta^{\theta v_1 + (1-\theta)v_2}C$ which does not lie in $ \Delta^{v_1}C$ (since $v_1 \neq \theta v_1 + (1-\theta)v_2$). For this vector, strict equality must hold, since $\dist^2(u, \Delta^{\theta v_1 + (1-\theta)v_2}C) = 0 < \dist^2(u, \Delta^{v_1}C)$. Now, since the distance from a vector $u$ to a convex set is continuously dependent on $u$, the strict inequality even holds in some $\epsilon$-ball around $u$, which has positive Gaussian measure. The two expressions are hence not almost surely equal. This is a contradiction, and hence $\phi$ is strictly convex.

	\emph{Continuity:} We have, for $u \in \R^d$ fixed and each $p \in C$
	\begin{align*}
		\dist (u,  \Delta^v C) \leq \norm{u - \Delta^v p}_2 \leq \norm{u-\Delta^{\hat{v}}p}_2 + \sum_{p \in C} \norm{\Delta^{\hat{v}}p-\Delta^{v}p}_2
	\end{align*}
	\begin{align*}
		\vert\dist (u,  \Delta^v C) - \dist(u, \Delta^{\hat{v}}C)\vert & \leq \sup_{p \in C} \norm{(\Delta^v-\Delta^{\hat{v}}) p}_2 \leq  
		\sup_{p\in C}\norm{\Delta^{v-\hat{v}}p}_2 \\
		 &= \sup_{p \in C} \sqrt{\sum_{i=1}^d(v(i)-\hat{v}(i))^2 p(i)^2 }\leq \norm{v - \hat{v}}_2.
	\end{align*}
	Denote 
	$\dist(g, \Delta^v C)=r_v$ and $\dist(g, \Delta^{\hat{v}} C)=r_{\hat{v}}$, respectively. The inequality just proved then reads $\abs{r_v-r_{\hat{v}}} \leq \norm{v - \hat{v}}_2$ and we may estimate
	\begin{align*}
		\abs{\phi(v) -\phi(\hat{v})}
		&=\abs{ \erw{ (r_v - r_{\hat{v}})(r_v+r_{\hat{v}})} } \leq \sqrt{\erw{\left(r_v-r_{\hat{v}}\right)^2}}\sqrt{\erw{\left(r_v+r_{\hat{v}}\right)^2}}	\\
		 &\leq \norm{v - \hat{v}}_2 \sqrt{2\erw{r_v^2+ r_{\hat{v}}^2}} 	= \norm{v - \hat{v}}_2 \sqrt{2(\phi(v)+\phi(\hat{v}))}.	
	\end{align*}
	$\phi$ is furthermore locally bounded, since $\phi(v) = \erw{\dist^2(g, \Delta^v C)} \leq 2\erw{\norm{g}^2 + \sup_{s\in C} \norm{D^vs}_2^2} = 2 + 2\norm{v}_2^2$. Hence, $\phi$ is continuous.
	
	\emph{Coercivity:} We have to prove that $\norm{v}_2 \to \infty$ implies $\phi(v) \to \infty$. Let us first note that due to our global assumption  $\alpha_i>0$ for each $i$, there exists for every $i=1, \dots, k$ an index $j_i \in S_i$ with $j_i \in \supp \ x_0$. For each such index and each $p \in C$, we have due to $p \vert_{\supp \ x_0} = \sgn \ x_0$ 
	\begin{align*}
		\inf_{p \in C} \norm{\Delta^v p}_2^2 \geq \inf_{p\in C} \sum_{i=1}^k v_i^2 p(j_i)^2 = \sum_{i=1}^k v_i^2 = \norm{v}_2^2.
\end{align*}	 
That implies that for any given $K>0$, there exists an $R>0$ so that $\norm{v}_2\geq R \Rightarrow \inf_{p \in C} \norm{\Delta^v p}_2^2 \geq K$. Hence, if $\norm{v}\geq R$ and $\norm{u}\leq r \leq K$, we have
\begin{align*}
	\dist^2(u, \Delta^v C) \geq \pos\big( \inf_{p \in C} \norm{\Delta^v p}_2 &- \norm{u}_2\big)^2 \\
	& \geq (K-r)^2. 
	\end{align*}
	 This implies that if $\norm{v}_2 \geq R$,
\begin{align*}
	\erw{\dist^2(g, \Delta^v C)}^2 &\geq \erw{\1_{\set{\norm{g} \leq r}}\dist^2(g, \Delta^v C)}^2 \\
	&\geq \prb{ \norm{g}\leq r} (K-r)^2.
\end{align*}
Since $\prb{ \norm{g}\leq r}>0$, this proves the claim. 

	Now it remains to prove that the minimum is not attained in $w=0$. For this, consider choosing all weights equal to $\lambda>0$, i.e $\omega=\lambda (1,\dots,1)$. Lemma \ref{lem:distSubdiffL1Weight} applied for $\omega_i=\lambda$  for all $i$ yields
	\begin{align*}
		\phi(\lambda (1,\dots,1))&= \sigma + \lambda^2 \sigma + (1-\sigma) \varphi(\lambda).
	\end{align*}
	We may estimate $\varphi(\lambda)$ by
	\begin{align*}
	\varphi(\lambda)=\sqrt{\frac{2}{\pi}}\int_\lambda^\infty (x-\lambda)^2 \exp(-x^2/2)dx  & \leq\sqrt{\frac{2}{\pi}} \int_0^\infty (x^2 - 2x\lambda +\lambda^2) \exp(-x^2/2)dx = \left(1 + \lambda^2 - \lambda \sqrt{\frac{2}{\pi}}\right).
	\end{align*}
	Therefore
	\begin{align*}
		\phi(\lambda (1,&\dots,1))= \sigma + \lambda^2 \sigma + (1-\sigma) \varphi(\lambda) \leq  \sigma + \lambda^2 \sigma +(1-\sigma)\left(1 + \lambda^2 - \lambda \sqrt{\frac{2}{\pi}}\right)<1 = \phi(0)
	\end{align*}
for small values of $\lambda$, since $\sigma<1$ by assumption.
	\end{proof}
	Having established an abstract result on the existence of optimal weights $\omega$, we now take a step back to see that the proof of said theorem actually provides a way to concretely calculate them -- we only need to minimize the function $\phi(v) = \erw{\dist^2(g, \Delta^v C)}$. Writing $v=\tau \omega$ as in the proof, we obtain, with the help of Lemma \ref{lem:distSubdiffL1Weight}
	\begin{align*}
	\phi(v)=\sigma + \sum_{i=1}^k \rho_i\left(\alpha_i v_i^2  + (1-\alpha_i) \varphi(v_i)\right).
	\end{align*}
	Since $\phi$ has the structure of a sum of convex functions, it can further be minimized by considering one variable at a time.  One could argue that this would have been a much simpler way of proving the assertion of Theorem \ref{prop:optimalWeight}, but the approach used in that proof has more potential to be applied to more general problems, and hence has an interest of its own.
	
	Since the convex functions are even differentiable, it is not hard to write down a formula for calculating $\omega_i=v_i$ analytically.  As a matter of fact, this has already been carried out in \cite[Prop 4.5]{AmelunxLotzMcCoyTropp2014}. There, the formula was used in the process of calculating the statistical dimension of $\mathcal{D}(\norm{\cdot}_1, x_0)$. In particular, it was not recognized as an optimal choice of a weight $\omega_i$. For completeness, we include the statement.
	
	\begin{cor} \label{cor:optWeightCalc}
		The optimal weights $\omega^*_i$ described in Theorem \ref{prop:optimalWeight} are given through the equations
		\begin{align*}
			\alpha_i \omega^*_i = (1- \alpha_i)\sqrt{\frac{2}{\pi}}\int_{\omega^*_i}^\infty (x-\omega_i)\exp(-x^2/2)dx, \quad  i = 1 \dots k.
		\end{align*}
		In particular, $\omega^*$ is independent of $\rho$.
	\end{cor} 
Finally, we can extract one more interesting corollary, which in fact is a counterpart of Theorem \ref{th:SynthesisUpperThresholds} in this setting. Using the fact that $\sigma = \sum_{i=1}^k \alpha_i \rho_i$, one can see that
	\begin{align*}
	\phi(v)=\sum_{i=1}^k \rho_i\left(\alpha_i (1+v_i^2 ) + (1-\alpha_i) \varphi(v_i)\right).
	\end{align*}
	Noticing that the term in the sum is actually the function we need to minimize in order to find $\mu_{\alpha_i \abs{S_i}, \abs{S_i}}$, we immediately arrive at the following result, which also summarizes the entirety of this paper.

	\begin{theo}\label{th:SynthesisThresholds}
		Let $T \sse [1, \dots, d]$ with $\abs{T}=s<d$, and a partition $\calS=(S_i)_{i=1}^k$ of $[1, \dots d]$ be given. Then there exist weights $\omega^*$ which minimize $\tilde{m}_{S,T,w}$ and which are unique up to multiplication with a positive scalar. Furthermore, we have
		\begin{align*}
			\tilde{m}_{\calS,T}=\sum_{i=1}^k\rho_i \mu_{\alpha_i \abs{S_i},\abs{S_i}},
		\end{align*}
		and the same also for the lower bounds $\tilde{n}_{S,T}=\tilde{m}_{\calS,T}-2/\sqrt{d}$. 
	\end{theo}

	\section{Numerical experiments}
	
	In this section, we present the results of some numerical experiments investigating the practical performance of the weighting strategy described above. Before presenting the set up of the experiment, let us note that it is numerically relatively easy to calculate the optimal weights described in Corollary \ref{cor:optWeightCalc}; we simply have to solve the $k$ independent, one-dimensional equations
	\begin{align*}
	\alpha_i \omega^*_i = (1- \alpha_i)\sqrt{\frac{2}{\pi}}\int_{\omega^*_i}^\infty (x-\omega_i)\exp(-x^2/2), \quad  i = 1 \dots k
	\end{align*}
	for $\omega_i$. For this, we used the MATLAB routine \texttt{fzero}. To be able to compare these weights to other strategies, we furthermore normalized them so that $\norm{\omega}_\infty=1$. In a similar manner, one can explicitly write down the optimality condition for $J_{\sigma, \alpha, \rho, \omega}'(\tau)=0$ to calculate the thresholds induced by given weights $\omega$. \newline
	
	In the first set of experiments, we consider the case of two sets $S_1, S_2$.  We fix the ambient dimension to $d=100$ and choose random $10$-sparse signals - the support of $x_0$ is drawn at random, and the values of $x_0$ on the support are  drawn from the standard multivariate normal distribution. Then we draw $3$, $5$ or $7$, respectively, indices in the support together with $7$, $5$ or $3$, respectively, indices outside the support to form a group $S_1$. The rest of the indices is then called $S_2$. Hence, in each experiment, $\rho_1=\sigma=0.1$ and $\alpha_1=0.3, 0.5$ and $0.7$, respectively. 
	
\begin{figure}
\center
 \includegraphics[scale=.4]{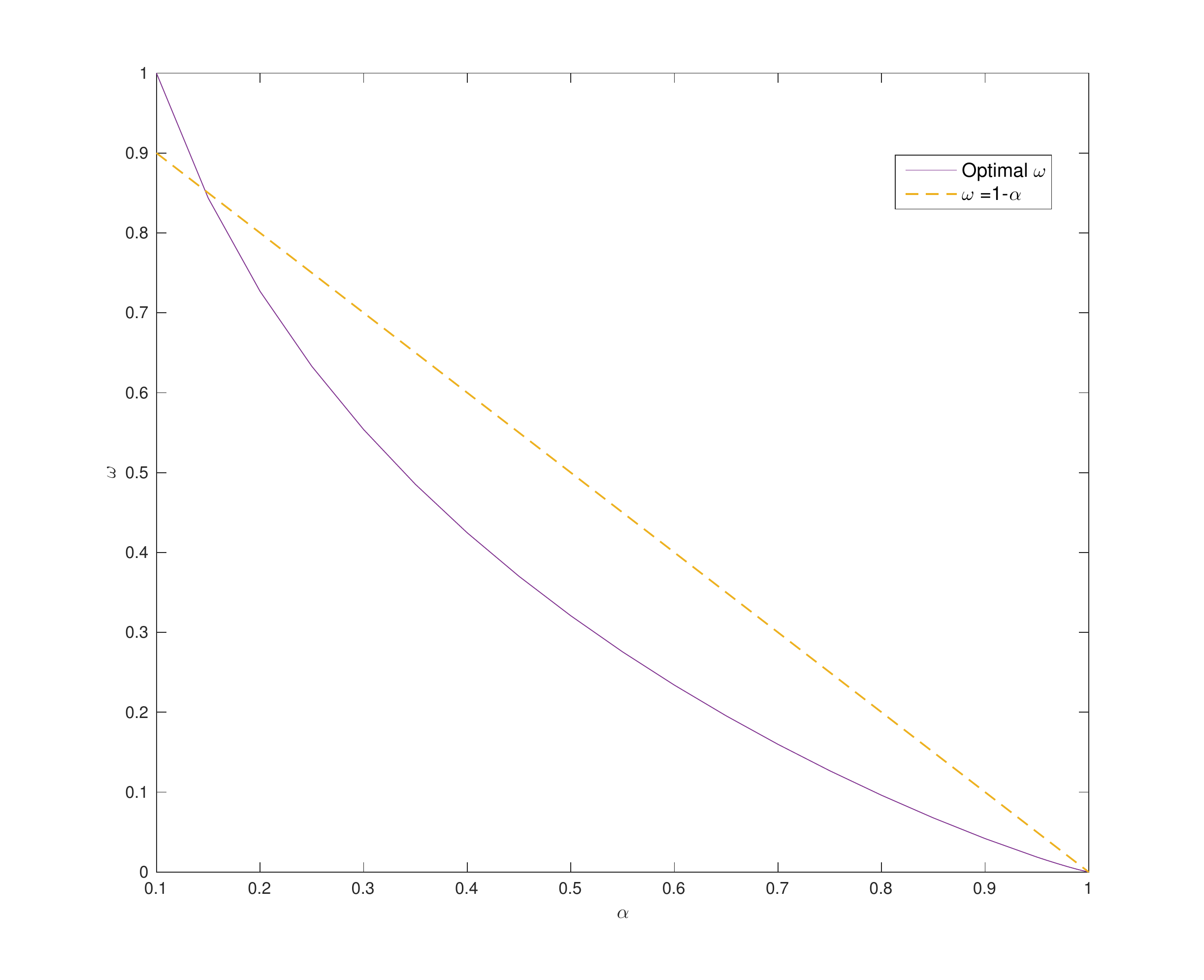}
 \caption{Plot of the optimal weight $\omega_1$ ($\omega_2$ is always equal to $1$) for $\rho=(0.1,0.9)$ for different $\alpha_1$ compared with $1- \alpha_1$. \label{fig:Weights}}
\end{figure}

\begin{figure}
\center
	\includegraphics[scale=.4]{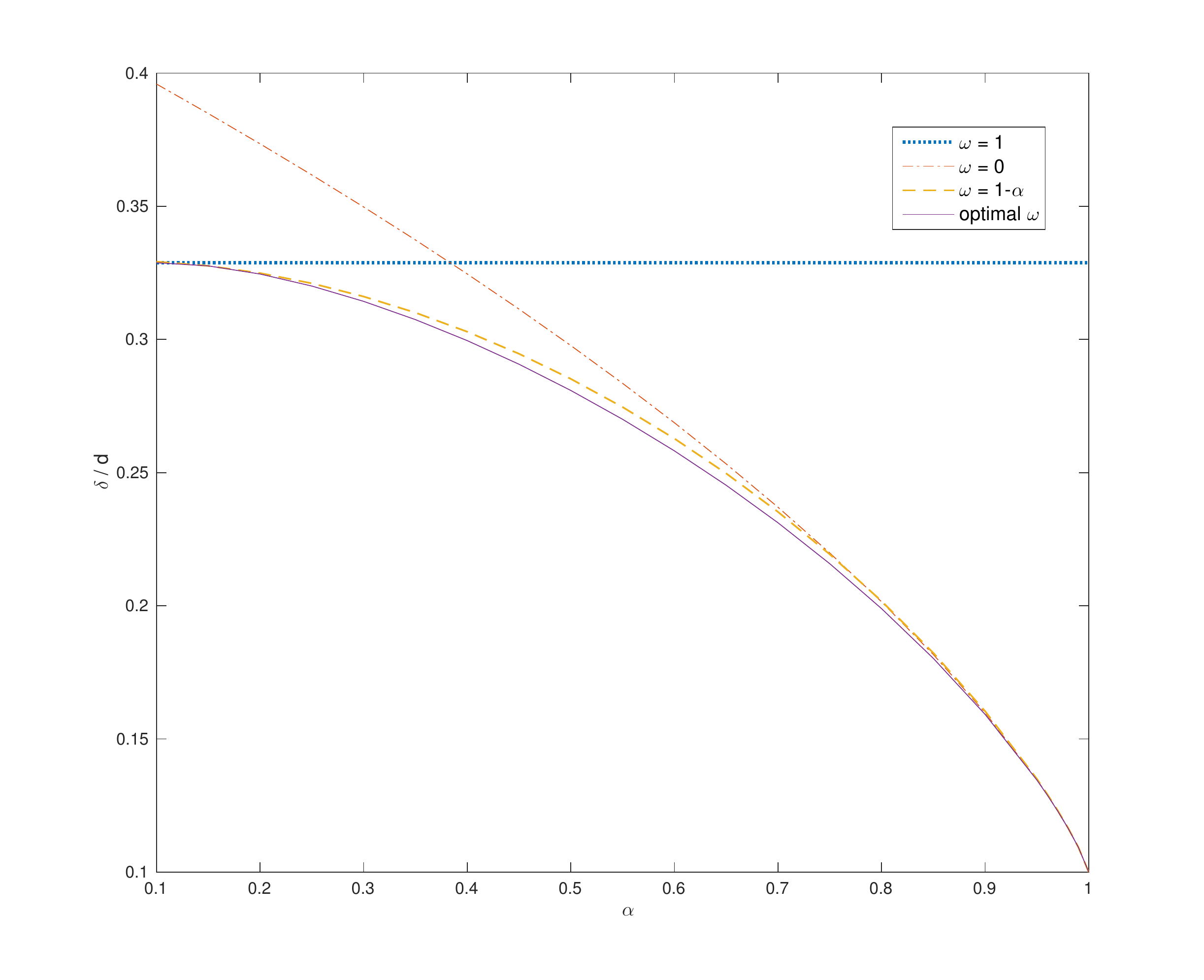}
	\caption{The threshold amount of Gaussian measurements needed to secure recovery with weighted $\ell_1$-minimization for different weighting strategies and $\rho=(0.1,0.9)$, depending on $\alpha_1$. \label{fig:Dims}}
\end{figure}	
	
	We consider $4$ strategies: 
	\begin{enumerate}
	\item $\omega=(1,1)$, which corresponds to standard $\ell_1$-minimization,
	\item the extreme strategy  $\omega=(0,1)$,
	\item  $\omega=(1-\alpha(1),1)$, i.e. the  one from \cite{MansourSaab2014},
	\item  the one with weights calculated as proposed in this work. 
	\end{enumerate}
	
	The resulting weights for strategy $(4)$ are depicted in Figure \ref{fig:Weights} for different values of $\alpha$. The theoretical thresholds $\inf_{\tau>0} J_{\sigma, \rho, \alpha,\omega}(\tau)$ were also calculated and are depicted in Figure \ref{fig:Dims}. As can be seen, the statistical dimensions for the our choice are in all cases lower compared to the other approaches. The strategy from \cite{MansourSaab2014} is however very close to the one described in this paper.
	
	To test the performance in practice, we for each $m \in [1, \dots, 35]$ draw a matrix $A\in \R^{m,d}$ according to the Gaussian distribution and solve the reweighted $\ell_1$-minimization problem  $(P_1)$ with help of the matlab package \texttt{cvx} \cite{cvx}. For each $m$, $1000$ experiments are performed, and a success is declared if the solution of the minimization problem differs no more than $0.001$ in $\ell_2$-norm from $x_0$. The results can be found in Figure \ref{fig:omegas}.
	
	The figures show that reweighted $\ell_1$-minimization performs significantly better using the weight chosen as in Theorem \ref{prop:optimalWeight} compared to standard Basis Pursuit ($\omega=1$) for all $\alpha$ tested. The same is true for the comparison to the strategy $\omega=0$ in the cases $\alpha=.3$ and $\alpha=.5$, but not in the case $\alpha_1=0.7$. This was expected, since experiments performed by the authors of \cite{FriedMansSaabYil2012} indicated that the choice of $\omega$ gets more important as $\alpha$ gets lower. The difference between our strategy and the one proposed in \cite{MansourSaab2014} is also present, however not significant. \newline

	\begin{figure} 
	\begin{multicols}{2}
	\centering
		\includegraphics[scale=.28]{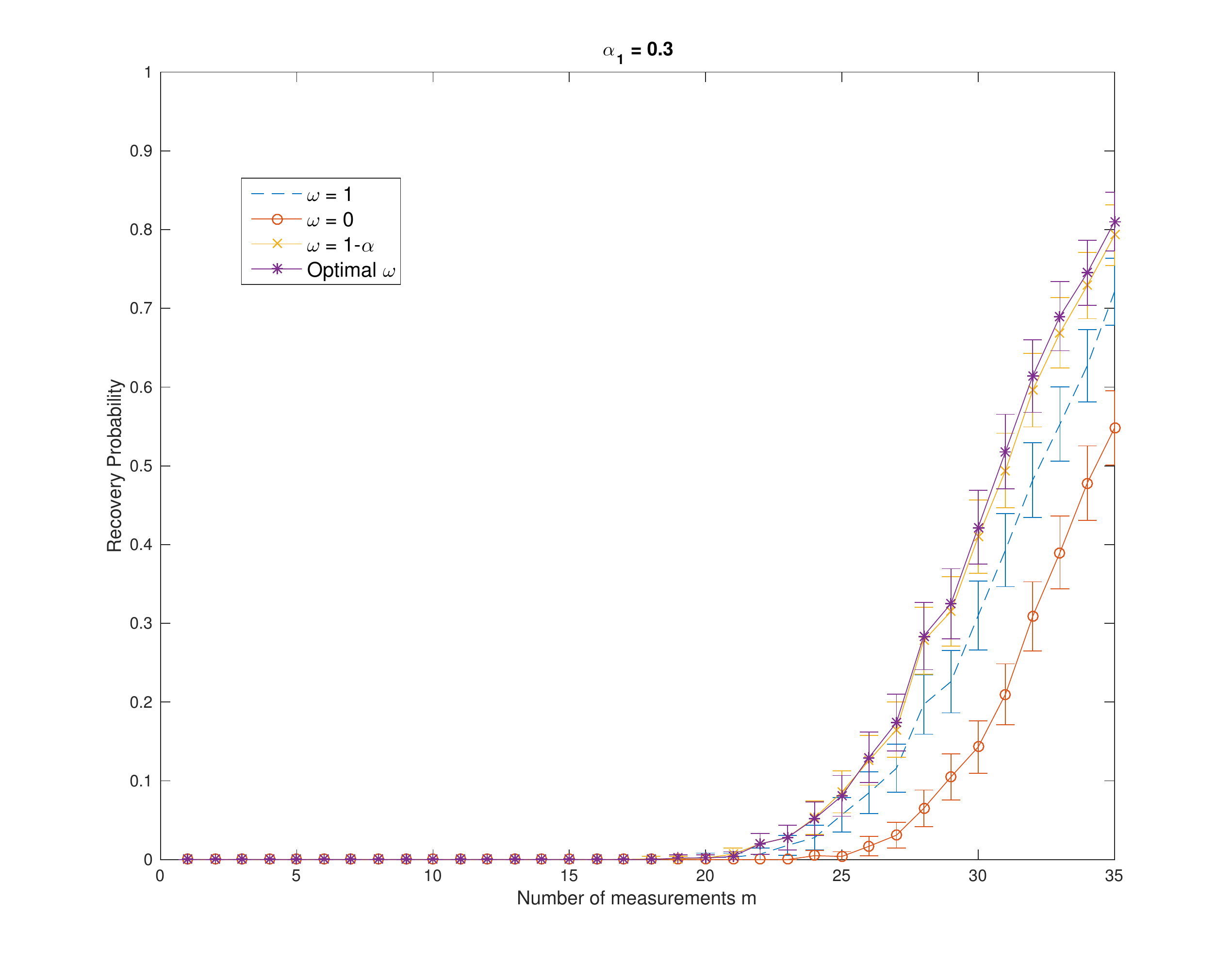}
	\includegraphics[scale=.28]{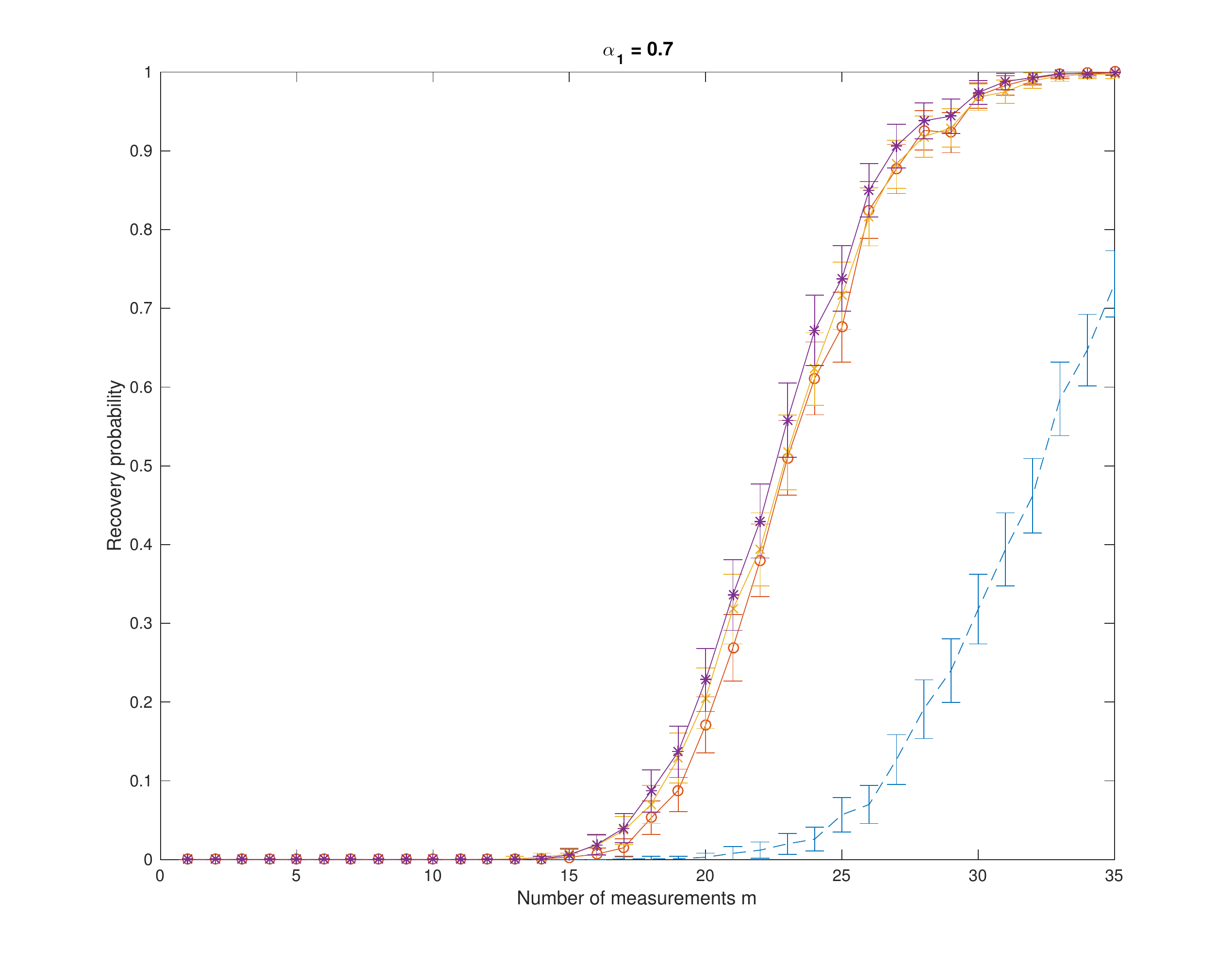}		
		
		\includegraphics[scale=.28]{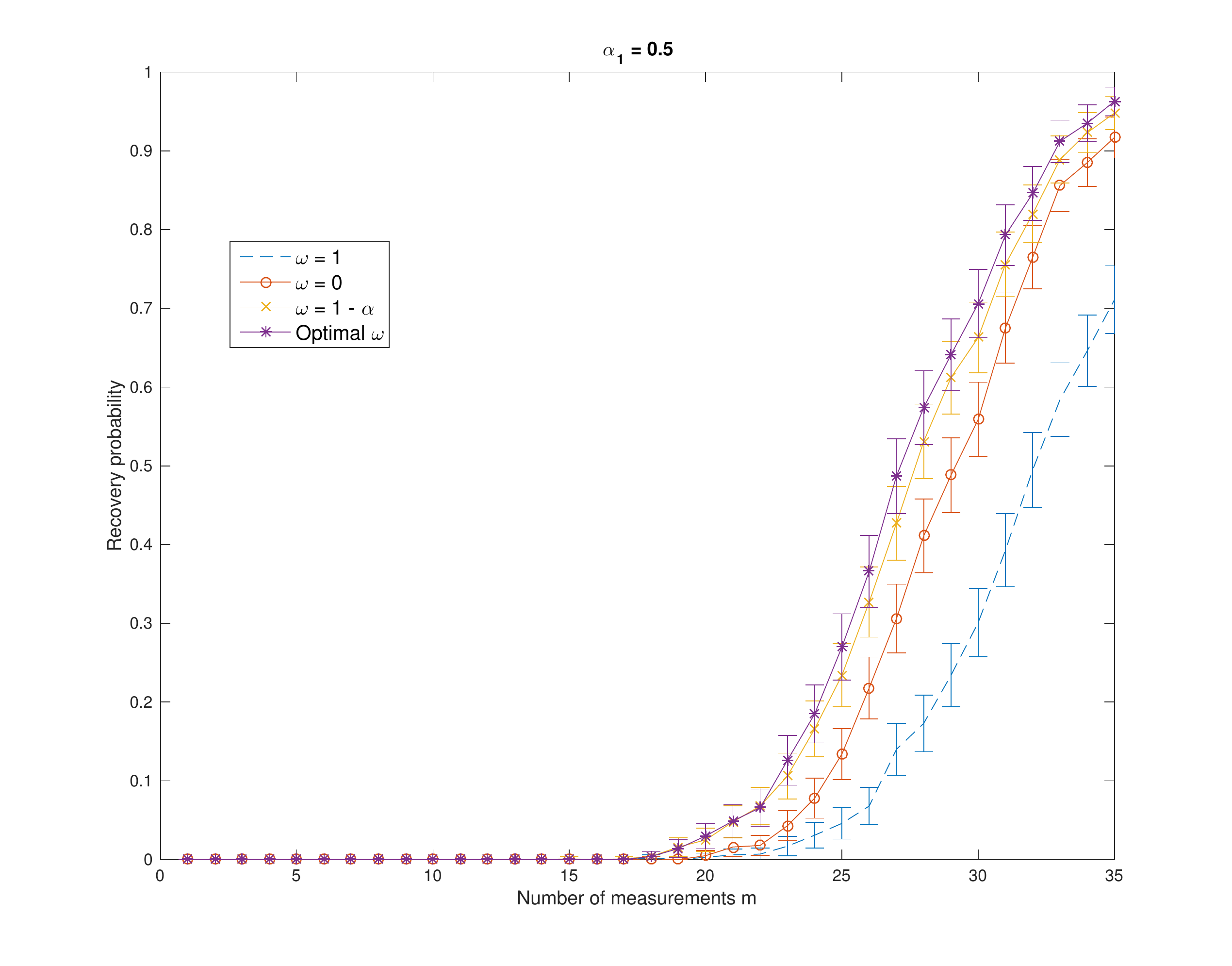}

		\caption{ Numerical experiments concerning the effect of optimally chooosing $\omega$. In each figure, $d=100$, $\sigma=\rho_1=0.1$. $\alpha_1$ is different in each figure -- $0.3$ in the upper left figure, $0.5$ in the upper right figure and $0.7$ in the lower figure, respectively. The calculated optimal weight $\omega_1$ ($\omega_2$ is always equal to $1$) in the respective experiments are equal to $0.5539$ $(\alpha_1=0.3)$, $0.3208$ $(\alpha_1=0.5)$ and $0.1599$ $(\alpha_1=0.7)$. The error bars are corresponding to 3 standard deviations.\label{fig:omegas}}
		\end{multicols}
	\end{figure}

	In order to see that also the theory for $k>2$ is practically relevant, we make another experiment and consider the case when $\alpha=(4/5,3/10,2/15,1/70)$ and $\rho=(.05,.1,.15,.7)$. $A$ is chosen as above for different values of $m$. We consider five different strategies
	\begin{enumerate}
		\item Set $\omega = 1$ (i.e. perform standard $\ell_1$-minimization.
		\item Consider the union $S_1 \cup S_2 \cup S_3$ as one region with $\rho=.3$ and $\alpha=9/30$, and choose the two weights as $\omega = (1- \alpha, 1)$.
		\item Consider the union $S_1 \cup S_2 \cup S_3$ as one region with $\rho=.3$ and $\alpha=9/30$, and choose the two weights as proposed in this work.
		\item Choose four weights, one for each $S_i$, with $\omega_4=1$ and $\omega_i = 1- \alpha_i$ for $i=1, \dots 3$.
		\item Choose four weights as proposed in this work.
	\end{enumerate}

	Note that although strategy $(4)$ bare resemblances to the one proposed in \cite{MansourSaab2014}, there are no theoretical results which motivate why it should be used. In particular, it was not proposed by the authors of that paper, since they only considered the case $k=2$. It should only be seen as a heuristic choice for comparison purposes. The results are depicted in Figure \ref{fig:ks}. We see that the optimal strategy proposed in this work involving four sets $S_i$ is perform significantly better than all of the other strategies.

	\begin{figure}
	\centering
	\begin{multicols}{2}
	\includegraphics[scale=.3]{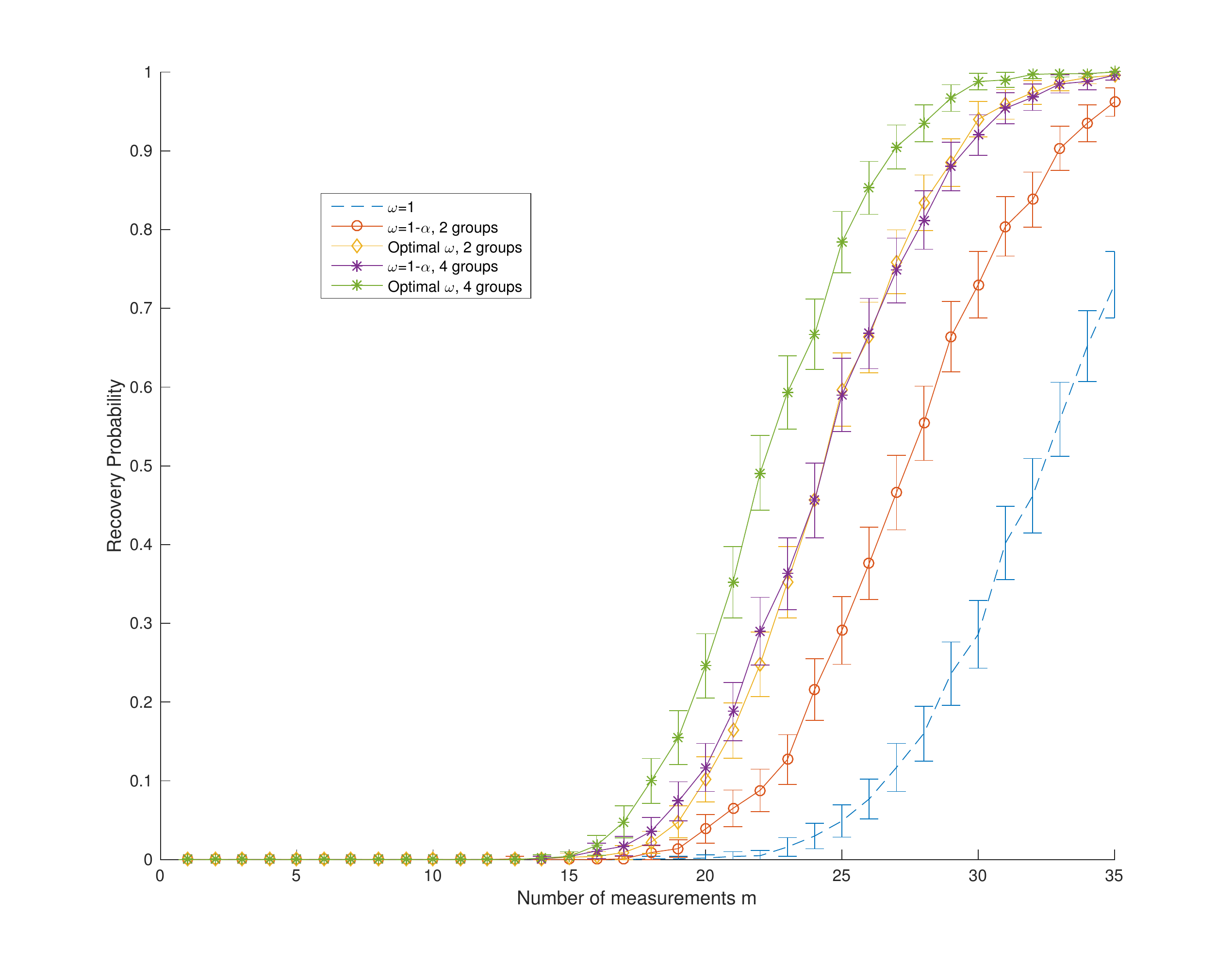}
	
	\caption{ Numerical experiment concerning the effect of using of a finer partition $S_1, S_2, S_3, S_4$ instead of forming two sets $S_1 \cup S_2 \cup S_3$ and $S_4$. The sparsity parameter is $\sigma=0.1$, $\rho=(0.05,0.1,0.15,0.7)$ and the ambient dimension is $d=100$. The optimal weights were in the first case calculated to $(0.0884,0.3742,0.5617,1)$ and in the second case $(.3742,1)$. The error bars are corresponding to 3 standard deviations. \label{fig:ks}}
	\end{multicols}
	\end{figure}

	\section*{Acknowledgment}
	\addcontentsline{toc}{section}{Acknowledgement}
	The author wishes to thank Prof. Dr.  G. Kutyniok for carefully proofreading of the paper as well as for fruitful discussions. He also wants to thank the anonymous reviewers for many useful comments. The author also acknowledges support from the Deutscher Akademischer Austauschdienst, DAAD.



%

%
%

\end{document}